%% file: root.tex
\documentclass[letterpaper, 10 pt, conference, epsf, psfrag]{ieeeconf}  % Comment this line out if you need a4paper

\IEEEoverridecommandlockouts                              % This command is only needed if 
\usepackage{graphicx, color}
\usepackage{fancyhdr}
\usepackage{amsmath}
\usepackage{epsfig} % for postscript graphics files
\usepackage{psfrag}
\usepackage{cite}
\usepackage{amsmath} % assumes amsmath package installed
\usepackage{amssymb}  % assumes amsmath package installed
\usepackage{dsfont}  % assumes amsmath package installed
\usepackage{todonotes}
\usepackage{tikz}
\usetikzlibrary{arrows,shapes,automata,backgrounds,petri,calc}

\usepackage{booktabs}
\usepackage[export]{adjustbox}

\usepackage{caption}
\usepackage{subcaption}
\newtheorem{proposition}{Proposition}
\newtheorem{lemma}{Lemma}
\newtheorem{theorem}{Theorem}
\newtheorem{remark}{Remark}
\newtheorem{definition}{Definition}

\DeclareMathOperator{\E}{\mathsf{E}}
\DeclareMathOperator{\T}{\mathsf{T}}

\DeclareMathOperator{\Tr}{tr}
\title{\LARGE \bf
 Optimal LQG Control under Delay-dependent Costly Information}

\author{Dipankar Maity$^{1}$, Mohammad H. Mamduhi$^{2,3}$, Sandra Hirche$^{3}$, Karl Henrik Johansson$^{2}$, and John S. Baras$^{1}$% <-this % stops a space
\thanks{}% <-this % stops a space
\thanks{$^{1}$D. Maity, and J. S. Baras are with the Department of Electrical \& Computer Engineering, Institute for Systems Research, University of Maryland, USA,
       {\tt\small} {\tt\small $\{$dmaity,baras$\}$@umd.edu}}
\thanks{$^{2}$M. H. Mamduhi, and K. H. Johansson are with the  Department of Automatic Control, The Royal Institute of Technology,
       SE-100 44 Stockholm, Sweden,
       {\tt\small} {\tt\small $\{$mamduhi,kallej$\}$@kth.se}}
\thanks{$^{3}$M. H. Mamduhi, and S. Hirche are with the Chair of Information-Oriented Control,
       Technical University of Munich, Arcisstra\ss{}e 21, D-80290 M\"unchen, Germany,
       {\tt\small} {\tt\small $\{$mh.mamduhi,hirche$\}$@tum.de}}
\thanks{This work is jointly supported by DARPA through ARO grant W911NF1410384, and by ONR grant N00014-17-1-2622, the German Research Foundation (DFG) within the Priority Program SPP 1914 ``Cyber-Physical Networking'', and the Knut and Alice Wallenberg Foundation, the Swedish Strategic Research Foundation, and the Swedish Research Council.}
}

\begin{document}

\maketitle
\thispagestyle{empty}
\pagestyle{empty}

%%%%%%%%%%%%%%%%%%%%%%%%%%%%%%%%%%%%%%%%%%%%%%%%%%%%%%%%%%%%%%%%%%
\input{abstract.tex}
\input{introduction.tex}                      
\input{problem.tex}
\input{optimality.tex}
%\input{communication.tex}
%\input{Optimal_Switching_Strategy.tex}
\input{Simulation.tex}
%\input{stability.tex}
\input{conclusions.tex}
\input{Discussion.tex}
\input{appendix.tex}

%%%%%%%%%%%%%%%%%%%%%%%%%%%%%%%%%%%%%%%%%%%%%%%%%%%%%%%%%%%%%%%%%%%%%%%%%%%%%%%
%\section*{ACKNOWLEDGMENT}
%This work is partly funded by the German Research Foundation (DFG) within the Priority Program SPP 1914 ``Cyber-Physical Networking''.

%%%%%%%%%%%%%%%%%%%%%%%%%%%%%%%%%%%%%%%%%%%%%%%%%%%%%%%%%%%%%%%%%%%%%%%%%%%%%%%%

\bibliographystyle{ieeetr}
\bibliography{LCSS}

%\begin{thebibliography}{99} 
%
%\bibitem{c1} Kozin, F."A survey of stability of stochastic systems." Automatica, vol. 5, pp. 95--112, 1969.
%
%\bibitem{c2} Hager, William W., Horowitz,  Larry L."Convergence and Stability Properties of the Discrete Riccati Operator Equation and the Associated Optimal Control and Filtering Problems." SIAM Journal Control and Optimization, 14(2), pp. 295--312, 1976.
%
%\end{thebibliography}

\end{document}

%% file: abstract.tex
\begin{abstract}
In the design of closed-loop networked control systems (NCSs), induced transmission delay between sensors and the control station is an often-present issue which compromises control performance and may even cause instability. A very relevant scenario in which network-induced delay needs to be investigated is costly usage of communication resources. More precisely, advanced communication technologies, e.g. 5G,  are capable of offering latency-varying information exchange for different prices. Therefore, induced delay becomes a decision variable. It is then the matter of decision maker's willingness to either pay the required cost to have low-latency access to the communication resource, or delay the access at a reduced price. In this article, we consider optimal price-based bi-variable decision making problem for single--loop NCS with a stochastic linear time-invariant system. Assuming that communication incurs cost such that transmission with shorter delay is more costly, a decision maker determines the switching strategy between communication links of different delays such that an optimal balance between the control performance and the communication cost is maintained. In this article, we show that, under mild assumptions on the available information for decision makers, the separation property holds between the optimal link selecting and control policies. As the cost function is decomposable, the optimal policies are efficiently computed. %We also show stability of the closed-loop system in mean-square sense under the obtained policies.
\end{abstract}

%% file: introduction.tex
%%%%%%%%%%%%%%%%%%%%%%%%%%%%%%%%%%%%%%%%%%%%%%%%%%%%%%%%%%%%%%%%%%%%%%%%%%%%%%%%
\section{INTRODUCTION}

In the design of closed-loop NCSs where information is exchanged between sensors, controller and actuator over a limited-resource communication network, induced transmission delay plays a key role in characterizing control performance and stability properties \cite{5400548,878601}. Day-by-day increase of data volume that needs to be exchanged %among different parts of a distributed system 
urges access to fast and low-error communication infrastructure to support the stringent real-time requirements of such systems. This, however, imposes higher communication and computation costs, resulting in reconsideration of employing time-based sampling techniques with equidistant fixed temporal durations. Various approaches are developed to coordinate data exchange in NCSs with the aim of reducing the total sampling and communication rate. Effective techniques such as event-based sampling, scheduling, and network pricing are introduced leading to the reduction of communication and computational costs by restricting unnecessary data sampling. Having intermittent sampling, delay is induced in various parts of the networked system which may degrade control performance. Hence, such decision makers need to be carefully designed in order to preserve stability as well as providing required quality-of-control (QoC) guarantees.

Event-based control introduced as a beneficial design framework to coordinate sampling of signals based on some urgency metrics, e.g. an action is executed only when some pre-defined events are triggered~\cite{bernhardsson1999comparison}. This idea received substantial attention and is further developed as a technique capable of significantly reducing sampling rate while preserving the required QoC~\cite{MAMDUHI2013356,5510124,6068223,7300689,MAMDUHI2017209}. The mentioned works, among many more, consider sporadic data sampling governed by real-time conditions of the control systems or the communication medium. 
Synthesis of optimal event-based strategies in NCSs is also addressed \cite{RabiBaras2006,CervinCamacho2011,MolinHirche2013}%SinopoliSastry2005,

Data scheduling is employed by communication theorists for decades as an effective resource management technique \cite{Bao:2001:NAC:381677.381698,LiYaxin2001}. By emerging NCSs as integration of multiple control systems supported by communication networks, cross-layer scheduling attracted more attentions. The reason is scheduling induces delay and affects NCS stability and QoC, hence, scheduling approaches that take into account real-time conditions of control systems become popular \cite{WalshYe2001,MAMDUHI2017209,6286997}. 

Designing price mechanisms for multi-user networks, to guarantee quality-of-service (QoS) is popular  in communication \cite{BasarSrikant2002,1208958}. %For example, an auction-based pricing mechanism is proposed in \cite{1208958} for DiffServ-based networks such that network manager assigns ``quality'' to the users according to their bids. 
In these works the goal is often set to maximize the QoS, which is a network-dependent utility expressed often in form of effective bandwidth requirements. In NCSs, however, QoC is of interest which additionally takes into account users' dynamics. Optimal communication pricing aiming at maximizing the QoC in NCSs has received less attention with a few exceptions, e.g. \cite{MolinHirche2014TAC,1664999}.

In those mentioned works, delay is considered an inevitable network-induced phenomena resulting from the employed sporadic sampling mechanisms. %, which affects stability margin and control performance. 
Novel communication technologies, e.g. 5G, offer not only ``bandwidth'' as the resource to pay for, but also real-time ``latency''. Users can decide to pay a higher price for lower latency or to delay data exchange at a reduced price. %The same idea is also applied on sensor networks where data can be acquired from different sensors each with particular accuracy or resolution and more accurate data can be acquired at the expense of higher computational cost. 
In such scenarios, the resulting induced delay plays as an explicit decision variable, i.e., users can optimize their utilities versus the communication price. In this article, we take the first steps in this direction by addressing the problem of joint optimal control and delay-dependent switching policies for a single-loop NCS with costly communication. The switching law determines the length of delay associated with the data sent over the network. We assume that every transmission incurs a cost determined by the associated delay, such that shorter delay %between the time of sending data and the time of it being received, 
incurs higher cost. Aggregating the LQG cost and delay-dependent communication cost over a finite horizon, we derive the optimal control and switching laws assuming that communication prices are known \textit{apriori}. It is then shown that the optimal control and switching laws are separable in expectation, and thus can be computed offline.~It guarantees the computational feasibility of our proposed approach.

%In the rest of this article, we state the problem of interest in Section~\ref{prob_state}. Section~\ref{optimal_design} presents the joint optimal control and transmission laws. Simulations are shown in Section~\ref{sim_results}.%, and proofs of the theorems are provided in the Appendix. %in Section~\ref{Appendix}.

%% file: problem.tex
\section{Problem Formulation}\label{prob_state}
Consider an LTI controlled system, consisting of a physical plant $\mathcal{P}$ and a controller $\mathcal{C}$. The plant $\mathcal{P}$ is descried by %the following stochastic difference equation
\begin{equation}
\begin{aligned} 
x_{k+1}=Ax_k+Bu_k+w_k
\end{aligned}
\label{eq:sys_model}
\end{equation}
where $x_k\!\in \!\mathbb{R}^{n}$ is the system state, $u_k \!\in \!\mathbb{R}^{m}$ is the control signal executed at time $k$, and $w_k\in \mathbb{R}^{n}$ is the exogenous disturbance. The constant matrices $A\!\in \!\mathbb{R}^{n\times n}$, and $B\!\in \!\mathbb{R}^{n\times m}$ describe drift matrix, and input matrix, with the pair $(A,B)$ assumed to be controllable. The disturbance $w_k$ and the initial state~$x_0$ are i.i.d. random variables with realizations $w_k\!\sim \!\mathcal{N}(0,W)$, and $x_0\!\sim\!\mathcal{N}(0,\Sigma_0)$, where $W\!\succeq \!0$  and $\Sigma_0\!\succeq \!0 $ denoting the variances of the respective Gaussian distributions. %The initial state~$x_0$ is drawn from a Gaussian distribution $\mathcal{N}(0,\Sigma_0)$. 
For the purpose of simplicity, we assume that the sensor measurements are perfect copies of state values. %, i.e. the output matrix is unity, and there exists no measurement noise.

In this article we address a delay-dependent LQG problem. As shown in Fig. \ref{fig:sysmodel}, there are $D$ number of links, each associated with a delay time $\{1,\ldots,D\}$. Selection of transmission link decides the arrival time instance of data at the controller, i.e. controller update may be delayed. Each link has a known cost of operation that increases as delay decreases.  %These costs for the delay-links are designed by another agent (think of a network manager that own the communication channel) and, 
%In this work, we assume the costs are known. %The strategic pricing policy of the network is an equally interesting subject to address, which is excluded from this work due to space limitations.%, it will be studied elsewhere. 
Note that, in the NCS scenario illustrated in Fig. \ref{fig:sysmodel}, %is different from the classical NCS as the delay selection is not left to the sensors rather we have 
the control unit which determines the switching policy of transmission links is a separate decision making unit with specific information structure, and must be distinguished from the plant controller.

Recall that the optimal LQG control is a certainty equivalent control that uses state estimation based on regularly-sampled measurements. In this work, however, the arrival of measurements and consequently the estimation quality depends on the selected delay link. Thus, unlike standard optimal LQG  where there is one controller that generates the control signal, here another control unit with an appropriate information structure exists and determines the optimal strategy to select the delay links. To take this into account, we first define the binary decision variable $\theta^i_k$ as follows:
\begin{equation*}
\theta^i_k=\begin{cases} 1, &\!\!\!\!\!\!\!\!\!\!\!\text{at time $k$ link with $i$ step delay is selected} \\ 0, \qquad & \!\!\!\!\!\!\!\!\!\!\!\text{at time $k$ link with $i$ step delay is not selected} \end{cases}
\end{equation*}
Based on the above definition, if $\theta_k^i=1$, the controller has access to system state $x_k$ at time-step $k+i$. 
We assume the possibility of selecting more than one links at each time, i.e.
\begin{equation}\label{const1}
\sum\nolimits_{i=1}^D\theta^i_k\ge 1, \quad \forall \; k\in\{1,2,\ldots\}
\end{equation}
where, the finite variable $D\!\in \!\mathbb{N}$ denotes the maximum allowable delay. Each link with associated delay $i$ is assigned a price, denoted by $\lambda_i \!\in \!\mathbb{R}^+$, to be paid if it is selected for transmission. Hence, at each time $k$, the switching~decision $\theta_k$, can be represented by a binary-valued vector as follows
\begin{equation}\label{decision_variables}
\theta_k\triangleq[\theta^1_k,\ldots,\theta^D_k]^{\textsf{T}}.
\end{equation}
The prices for each communication link $i\in \{1,\ldots,D\}$ are denoted by $\lambda_i$, and are fixed \textit{apriori} with the following order:%way that the shorter the transmission delay is, the higher the price becomes. Therefore, the following condition holds:
\begin{equation*}
\lambda_1 >\lambda_2>\ldots>\lambda_D>0.
\end{equation*}
%The described system scenario is schematically depicted in Fig.\ref{fig:sysmodel}. 
%
%We, moreover, denote the set of transmission decision variables up to, and including, time-step $k$ as follows:
%\begin{equation}
%\Theta_{[1,k]}\triangleq\cup_{t=1}^{k}\theta_t.
%\end{equation}

\begin{remark}
In this framework, a link with very large delay~$D_{ol}\!\gg~\!1$ and cost $\lambda_{ol}\!=\!0$ can be added such that a transmission becomes very unlikely. Theoretically, $D_{ol}\!\rightarrow \!\infty$, the system opts to be open-loop. In our scenario, however, system is forced to select at least one link, according to  (\ref{const1}).
\end{remark}
\begin{figure}[tb]
\centering
\psfrag{a}[c][c]{\scriptsize \textsf{Plant}}
\psfrag{b}[c][c]{\scriptsize \textsf{Controller}}
\psfrag{v}[c][c]{\scriptsize \textsf{link with $d$}}
\psfrag{w}[c][c]{\scriptsize \textsf{step delay}}
\psfrag{x}[c][c]{\scriptsize $Z^{\text{-}d}(x_k)$}
\psfrag{e}[c][c]{\scriptsize $\theta_k$}
\psfrag{d}[c][c]{\scriptsize $u_k$}
\psfrag{g}[c][c]{\scriptsize $x_k$}
\psfrag{h}[c][c]{\tiny $Z^{\text{-}1}(x_k), \!\lambda_1$}
\psfrag{k}[c][c]{\tiny $Z^{\text{-}2}(x_k), \!\lambda_2$}
\psfrag{f}[c][c]{\tiny $Z^{\text{-}D}(x_k), \!\lambda_D$}
\psfrag{i}[c][c]{\scriptsize $\cdot$}
\psfrag{j}[c][c]{\scriptsize $\cdot$}
\psfrag{l}[c][c]{\scriptsize $\cdot$}
\psfrag{c}[c][c]{\scriptsize \textsf{Link decision unit}}
\psfrag{m}[c][c]{\scriptsize $Z^{\text{-}d}(x_k)$}
\psfrag{y}[l][l]{\scriptsize \textsf{delay switch}}
 \vspace{-1mm}
\includegraphics[width=0.46\textwidth,height=4.9cm]{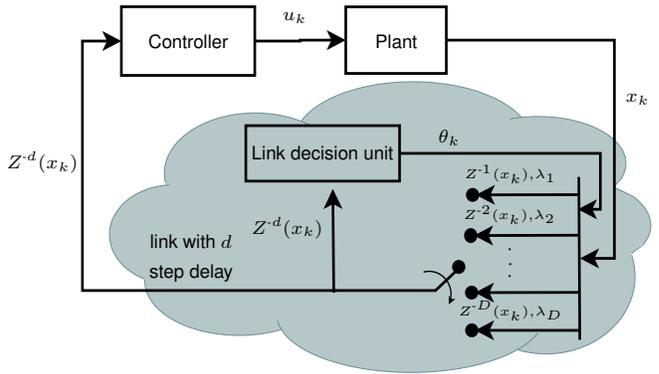}
%\psfragfig[width=0.38\textwidth,height=4.15cm]{Joint_delayed.eps}
 \caption{Schematic of a closed-loop system with communication delay, where $Z^{-d}(x_k)$ means $x_k$ will be received by control unit $d$ time-steps later, at the expense of $\lambda_d$.}
 \label{fig:sysmodel}
 \vspace{-5.mm}
\end{figure}

%\begin{figure}
%\centering
%\includegraphics[width=0.45 \textwidth]{Schematic1.png}
%%\psfragfig[width=0.38\textwidth,height=4.15cm]{Joint_delayed.eps}
% \caption{Schematic of a closed-loop system with communication delay, where $Z^{-d}(x_k)$ means $x_k$ will be received by control unit $d$ time-steps later, at the expense of $\lambda_d$.}
% \label{fig:sysmodel}
%\end{figure}

According to (\ref{decision_variables}), the received state information at the controller at time-step~$k$, denoted by $\mathcal{Y}_k$, is expressed as
\begin{equation}
\mathcal{Y}_k=\{\theta_{k-1}^1x_{k-1},\theta_{k-2}^2x_{k-2},\ldots, \theta_{k-D}^Dx_{k-D}\}
\end{equation}
where, $\theta_{-1}^i\!=\!\ldots\!=\!\theta_{-D}^i\!=\!1$, for all $i$ to represent equations compactly. 
%It should be noted that the controller uses the most recent received state information to compute the control signal $u_k$, and will disregard the rest, if any. 
The system possesses two decision makers; one decides the delay link via $\theta_k$, and one computes control signal $u_k$. To define the information set and the associated $\sigma$-algebra available to each decision maker, %for generating $u_k$ and for selecting $\theta_k$, 
we first introduce two sets $Y_{k}\!\triangleq \!\{\mathcal{Y}_0,\ldots, \mathcal{Y}_k\}$, and $U_k\!\triangleq \!\{u_0,\ldots, u_k\}$, containing the received state information, and control signals, up to and including time $k$, respectively.
%\begin{align}\label{hist_set1}
%Y_k &\triangleq \{\mathcal{Y}_0,\ldots, \mathcal{Y}_k\}\\\label{hist_set2}
%U_k &\triangleq \{u_0,\ldots, u_k\}
%\end{align}
%Recalling the definitions (\ref{decision_variables})-(\ref{hist_set2}), 
We now define the information sets $\mathcal{I}_k$ and $\mathcal{\bar{I}}_k$ at time $k$, respectively accessible for the switching and the plant controllers, as follows:
\begin{align*}
 \mathcal{I}_k &\triangleq \{Y_{k-1},U_{k-1}, \cup_{t=1}^{k-1}\{\theta_t\}\}, \quad
\mathcal{\bar{I}}_k \triangleq \{\mathcal{I}_k, \mathcal{Y}_k,\theta_k\}.
\end{align*}
At every time $k$, the control and delay switching strategies are measurable functions of the $\sigma$-algebras generated by $\bar{\mathcal I}_k$, and $\mathcal{I}_k$, respectively, i.e., $u_k\!=\!g_k(\bar{\mathcal{I}}_k)$, and $\theta_k\!=\!s_k(\mathcal{I}_k)$.
The order of decision making in one cycle of sampling is as follows: $\cdots \!\rightarrow\!\mathcal{I}_{k}\!\rightarrow\! \theta_k \!\rightarrow \!\bar{\mathcal{I}}_k \!\rightarrow \!u_k \!\rightarrow \!\mathcal{I}_{k+1}\!\rightarrow\!\cdots$. In general, the computation of the optimal control $u_k^*$ requires the knowledge of the optimal $\theta_k^*$. However, we show later that, under the introduced information structures, $\theta_k^*$ can be computed offline, and hence computation $u_k^\ast$ will not require on-line update about $\theta_k^*$. A possible implementation of this protocol is to send the preference of selecting the delay link to a network manager (it is the communication service provider that offers different QoS (delay)) that, upon receiving the sensor data $x_k$, selects the preferred transmission link. 
%\begin{align*}
%u_k&=g_k(\bar{\mathcal{I}}_k)\\
%\theta_k&=s_k(\mathcal{I}_k)
%\end{align*}

%The cost function to be minimized is
%\begin{align}\nonumber
%&J(u,\theta)=\\\label{joint_cost1}
%&\E\!\left[\sum_{t=0}^{T-1}\left[x_t^{\textsc{T}}Q_1x_t+u_t^{\textsc{T}}Ru_t+\!\sum_{i=1}^D\theta_t^i\lambda_i\right]+x_T^{\textsc{T}}Q_2x_T\right]\!,
%\end{align}

The cost function, that is jointly minimized by the two decision variables $g_k(\bar{\mathcal{I}}_k)$, and $s_k(\mathcal{I}_k)$, consists of an LQG part and communication cost. Within the finite horizon, the average cost function is stated by the following expectation
\begin{align*}
&J(u,\theta)\!=\!\E\!\left[\sum\nolimits_{t=0}^{T-1}\!\left[x_t^{\top}Q_1x_t\!+\!u_t^{\top}Ru_t\!+\!\theta_t^\top\Lambda\right]\!+x_T^{\top}Q_2x_T\right]\!,
\end{align*}
where, $\Lambda\triangleq[\lambda_1,\ldots,\lambda_D]^\top$, $Q_1\!\succeq \!0$, $Q_2\!\succeq \!0$, and $R\!\succ \!0$. 

%% file: optimality.tex
\section{Optimal Strategies for Control \& Switching}\label{optimal_design}

The optimal control and switching strategies are the minimizing arguments of the latter average cost function,~i.e.
\vspace{-2.75mm}
\begin{equation}\label{optimizal_str}
(u^\ast, \theta^\ast)=\text{arg}\min_{u,\theta} J(u,\theta),
\end{equation}
where the average cost optimal value equals $J^\ast\!=\!J(u^\ast, \theta^\ast)$. In the sequel, we show that the problem (\ref{optimizal_str}) is separable in its arguments $u$, and $\theta$ and can be disjointly optimized offline. In fact, we show that the optimal control policy is linear, and independent from the sequence of link switching decisions $\theta$, while the state estimation is a nonlinear function of $\theta$, which can be found offline. 
\begin{proposition}
An optimal strategy $u^*,\theta^*$ always exists if the Riccati equation $P_k$ has a well defined solution for the whole horizon $[0,T]$.
\begin{align}\label{Riccati}
&P_k=\\\nonumber
&Q_1+A^\textsf{T}\!\left(\!P_{k+1}-P_{k+1}B\left(R+B^\textsf{T}P_{k+1}B\right)^{-1}\!B^\textsf{T}P_{k+1}\!\right)\!A,\\
&P_T=Q_2
\end{align}
\end{proposition}

\begin{remark}
The existential condition of an optimal strategy is no different than the condition for a classical LQG optimal control problem.
\end{remark}

\subsection{Optimal control strategy}
Knowing that $\mathcal{I}_k\subseteq \bar{\mathcal{I}}_k$, we can re-write $J(u,\theta)$ as:
\vspace{-1mm}
\begin{align}\label{joint_cost}
&J(u,\theta)=\\\nonumber
&\E\!\left[\E\!\left[\!\E\!\left[\sum_{t=0}^{T-1}\!\left[\!x_t^{\top}Q_1x_t\!+\!u_t^{\top}Ru_t\!+\!\theta_t^\top\Lambda\right]\!+\!x_T^{\top}Q_2x_T\big|\bar{\mathcal{I}}_0\!\right]\!\big|\mathcal{I}_0\!\right]\!\right]
\end{align}
Thus, using the fact that $u_k$ and $\theta_k$ are $\bar{\mathcal{I}}_k$ and $\mathcal{I}_k$ measurable: 
\begin{align*}
\min_{\substack{u_{[0,T-1]}\\\theta_{[0,T-1]}}} J(u,\theta)=\E\left[\min_{\theta_{[0,T-1]}}\E\left[\min_{u_{[0,T-1]}}\E\left[C_0(u,\theta)|\bar{\mathcal{I}_0}\right]|\mathcal{I}_0\right]\right]
\end{align*}
where, $C_k(u,\theta)\!=\!\sum_{t=k}^{T-1}[x_t^{\top}Q_1x_t\!+\!u_t^{\top}Ru_t\!+\!\sum_{i=1}^D\theta_t^i\lambda_i]\!+\!x_T^{\top}Q_2x_T$. Moreover, we define the cost-to-go $J_k^*$ as follows:
\begin{align*}
J_k^*=\min_{\theta_{[k,T-1]}}\E\left[\min_{u_{[k,T-1]}}\E\left[C_k(u,\theta)|\bar{\mathcal{I}_0}\right]|\mathcal{I}_0\right],
\end{align*}
which reduces the optimization problem to the compact form
\begin{align*}
\min_{u_{[0,T-1]}, \;\theta_{[0,T-1]}} J(u,\theta)=\E\left[J_0^*\right].
\end{align*}

It then follows that $C_k(u,\theta)=V_k(u)+\sum_{t=k}^{T-1}\sum_{i=1}^D\theta_t^i\lambda_i$, where $V_k(u)=\sum_{t=k}^{T-1}\left[x_t^{\top}Q_1x_t+u_t^{\top}Ru_t\right]\!+\!x_T^{\top}Q_2x_T$. It is easy to verify that $V_k^*\!=\!\min_{u_{[k,T-1]}}\E\left[V_k(u)|\bar{\mathcal{I}}_k\right]$ is a standard LQG cost-to-go. Having this, we state Theorem \ref{thm:separation}.%, where the proof is provided in the Appendix \ref{app}:

\begin{theorem}\label{thm:separation}
Given the information set $\bar{\mathcal{I}}_k$, the optimal control policy $u^\ast_k=g^*_k(\bar{\mathcal{I}}_k)$, $k\!\in\!\{0,\ldots,T\!-\!1\}$, which minimizes $\E[V_k(u)|\bar{\mathcal{I}}_k]$, is a linear feedback law of the form
\begin{equation}\label{opt_u}
u^\ast_k=-(R+B^\textsf{T}P_{k+1}B)^{-1}B^\textsf{T}P_{k+1}A \E[x_k|\mathcal{\bar{I}}_k],
\end{equation} 
where, $P_k$ is the solution of the following Riccati equation:
\begin{align*}
&P_k\!=\!Q_1\!+\!A^\top\!(P_{k+1}\!-\!P_{k+1}B\!\left(\!R\!+\!B^\top \!P_{k+1}B)^{-1}\!\!B^\top \!P_{k+1}\right)\!A,\\
&P_T=Q_2. \nonumber
\end{align*}
%\begin{align}\nonumber
%P_T=Q_2.
%%\tilde{A}&=A-B\left(R+B^\textsf{T}P_{k+1}B\right)^{-1}B^\textsf{T}P_{k+1}A,\\\label{B_tilde}
%%\tilde{B}&=\left(R+B^\textsf{T}P_{k+1}B\right)^{-1}B^\textsf{T}P_{k+1}A.
%\end{align}
Moreover, the optimal cost is $V_k^*=\E[x_k|\mathcal{\bar{I}}_k]^{\T}P_k\E[x_k|\mathcal{\bar{I}}_k]+\pi_k$, where for all $T>t\ge k$, $\pi_k$ is expressed as 
%\begin{align*}
%V_k^*=\E[x_k|\mathcal{\bar{I}}_k]^{\T}P_k\E[x_k|\mathcal{\bar{I}}_k]+\pi_k.
%\end{align*}
\begin{align}\nonumber
\pi_k&=\E\!\left[e_k^\textsf{T}{P}_ke_k\!+\!\sum\nolimits_{t=k}^{T-1}e_t^\textsf{T}\tilde{P}_te_t|\mathcal{\bar{I}}_k\right]\!+\!\!\sum\nolimits_{t=k+1}^T\!tr(P_tW)
\end{align}
with, $e_k=x_k-\E[x_k|\mathcal{\bar{I}}_k]$, $\tilde{P}_t=Q_1 + A^{\textsf{T}}P_{t+1}A-P_t$.
\end{theorem}

\begin{proof}
The proof is presented in Appendix \ref{app}.
%Let the LQG optimal value function $V^\ast_k$, at time-step $k$, be written as follows
%\begin{align}\nonumber
%&V^\ast_k=\\\label{value_function}
%&\min_{u_{[k,T-1]}}\E\!\left[\sum_{t=k}^{T-1}x_t^{\top}Q_1x_t+u_t^{\top}Ru_t+x_T^{\top}Q_2x_T|\mathcal{\bar{I}}_k\right].
%\end{align} 
\end{proof}

From Theorem \ref{thm:separation}, $g_k^*(\bar{\mathcal{I}}_k)\!=\!L_k\E[x_k|\mathcal{\bar{I}}_k]$, with $L_k$ independent of $\theta$. This allows us to design the control law offline, while the estimator is $\theta$-dependent (\textit{Proposition} \ref{Prop_1}).~This is intuitive as $\lambda_i$'s are assumed to be state and time-independent. State and time dependency in prices is the subject of future work. For example when $\lambda_i=x_k^{\T}\mu_i^1x_k+{\mu^2_i}^{\T}x_k+\mu^3_i$, the Riccati equation of $P_k$ in \eqref{Riccati} will depend on the parameters $\mu^1_i,\mu^2_i,\mu^3_i$.  In this work, we restrict ourselves to state-independent, time-independent costs for the links.

\subsection{Optimal Switching Strategy}\label{optimal_switch}
Here we first show that the estimation at the controller is $\theta$-dependent. It results in $e_k$ being also $\theta$-dependent, $\forall k\!>\!0$.%, we explicitly derive the dynamics of $e_t$ as function of $\theta_t$.

\begin{proposition} \label{Prop_1}
The estimator dynamics is $\theta$-dependent s.t.
\begin{equation}\label{estimation2_1}
\hat{x}_k=\E[x_k|\mathcal{\bar{I}}_k]=\sum\nolimits_{i=1}^{\min\{D,k+1\}} b_{i,k}\E[x_k|x_{k-i},U_{k-1}],
\end{equation}
where, $\forall k\!\ge \!0, i\!\in \!\{1,\cdots,D\}$, $b_{i,k}\!\in \!\{0,1\}$. Moreover, $\sum_{i=1}^{\min\{D,k+1\}}b_{i,k}\!=\!1$, and if $D\!>\!k$, then $\sum_{i=k+2}^Db_{i,k}\!=\!0$. 
%$\E[x_k|x_{-1},U_{k-1}]\triangleq\E[x_k|\mathcal{\bar{I}}_0,U_{k-1}]$
\end{proposition}

\begin{proof} Proof is presented in the Appendix \ref{app:2}
\end{proof}

%where, $\E[x_k|x_{k-j}]$ denotes the estimate of the system state $x_k$ at time-step $k$, given that the most recent received state information is $x_{k-j}$, from $j$ time-step earlier. In (\ref{estimation1}), the estimate is also conditioned on $\bar{I}_k$ which is dropped for the ease of notation. The condition The expression (\ref{estimation1}) confirms, as it was expected, the state estimation at the current time $k$ is a nonlinear function of decision variables $\theta_{k^\prime}^i$ for all $k^\prime<k$, and for all $i\in\{1,\ldots,D\}$. Define the new variable $b_{i,k-1}$, $i\in \{1,\ldots,D\}$, as follows:
%\begin{align*}
%b_{1,k-1}&=\theta_{k-1}^1,\\
%b_{2,k-1}&=\left(1-\theta_{k-1}^1\right)\!\left(\theta_{k-2}^1\!+\!\theta_{k-2}^2\right),\\
%b_{3,k-1}&=\left(1\!-\theta_{k-1}^1\right)\!\left(1\!-\theta_{k-2}^1\right)\!\left(1\!-\theta_{k-2}^2\right)\!\left(\!\sum_{i=1}^3\theta_{k-3}^i\!\right),\\
%&\quad \!\!\!\vdots\\
%b_{D,k-1}&=\prod_{d=1}^{D-1}\prod_{j=1}^{d}\left(1-\theta_{k-d}^j\right)\left(\sum_{i=1}^D\theta_{k-D}^i\right)
%\end{align*}
%Then, (\ref{estimation1}) can be re-expressed as
%\begin{equation}\label{estimation2}
%\E[x_k|\mathcal{\bar{I}}_k]=\sum_{i=1}^D b_{i,k-1}\E[x_k|x_{k-i}]
%\end{equation}

%Therefore, 
%\begin{equation}\label{err-state_1}
%e_k\triangleq x_k - \E[x_k|\mathcal{\bar{I}}_k].
%\end{equation}
Defining $\tau_k\!\triangleq\!\min\{D,k\!+\!1\}$, and initial condition $e_0\!=\!x_0\!-\!\E[x_0]$, and $w_{-1}\!=\!e_0$ for notational convenience; and
knowing the noise realizations $\{w_{-1},w_0,\cdots,w_{T-1}\}$ are mutually independent, it concludes from Proposition \ref{Prop_1}, that
%\begin{align}
%e_k=& \sum_{i=1}^{\min\{D,k+1\}} b_{i,k}\sum_{j=1}^{i}A^{j-1}w_{k-j}, \quad \forall k\ge 1.
%\end{align}
\begin{align}\label{err-state2}
e_k&=x_k \!- \!\E[x_k|\mathcal{\bar{I}}_k]=\sum\nolimits_{j=1}^{\tau_k}\sum\nolimits_{i=j}^{\tau_k}\!b_{i,k}A^{j-1}w_{k-j},
\end{align}
%where, $t_m\!\triangleq\!\min\{D,k\!+\!1\}$, and initial condition $e_0\!=\!x_0\!-\!\E[x_0]$, and we define $w_{-1}\!=\!e_0$ for notational convenience. %in the expression (\ref{err-state2}). %, hence $\{w_{-1},w_0,\cdots,w_{T-1}\}$ are mutually independent.
%It is worth mentioning that, at time-step $k$, only for one specific $i$ the variable $b_{i,k-1}$ equals one, and the other $b_{\tilde{i},k-1}$ terms, for $\tilde{i}\neq i$ and $\tilde{i}\in \{1,\ldots,D\}$, vanishes. 
%Thus, \eqref{err-state2} can be re-written as
%\begin{align*}
%e_k=\sum_{j=1}^{\min\{D,k+1\}}\sum_{i=j}^{\min\{D,k+1\}}b_{i,k}A^{j-1}w_{k-j}.
%\end{align*}

Defining $M_k=\E[e_ke_k^{\T}~|\bar{\mathcal{I}}_k]$, it is straightforward to show
\begin{align*}
M_k=&\sum\nolimits_{i=1}^{\tau_k}c_{i,k}{A^{i-1}}^\textsf{T}W_{k-i}A^{i-1},
\end{align*}
where $W_{-1}=\E[e_0e_0^{\T}]=\Sigma_0$, $W_0=W_1=\cdots=W$, and $c_{i,k}=\sum_{j=i}^{\tau_k}b_{j,k}$.
Having this, one can show
\begin{align*}
V_k^*&=\hat{x}_k^\textsf{T}P_k\hat{x}_k\!+\!tr(P_kM_k)\!+\!\sum_{t=k}^{T-1}tr(\tilde P_tM_t)\!+\!\!\sum_{t=k+1}^T\!\!tr(P_tW).
\end{align*}
%where, for all $k\geq 1$
%\begin{align*}
%M_k=&\sum_{i=1}^{D}c_{i,k-1}{A^{i-1}}^\textsf{T}WA^{i-1},\\
%M_0=&\Sigma_0 \nonumber
%\end{align*}
%and
%\begin{align*}
%c_{i,k-1}=\sum_{j=i}^Db_{j,k-1}.
%\end{align*}
Consequently, we can express $J^*_0$ as follows:
\begin{align} \label{LP}
J^*_0&\!=\!\min_{\theta_{[0,T-1]}}\!\sum\nolimits_{t=0}^{T-1}\!\!\left[\!\sum\nolimits_{i=1}^{\tau_t}\!c_{i,t} tr(\tilde P_t{A^{i-1}}^\textsf{T}W_{t-1}A^{i-1})\!+\!\theta_t^\textsf{T}\Lambda \!\right] \nonumber\\
%&+\sum_{i=1}^Dc_{i,-1} tr( P_0{A^{i-1}}^\textsf{T}W_{t-i}A^{i-1})
&+\hat{x}_0^\textsf{T}P_0\hat{x}_0+\sum\nolimits_{t=1}^Ttr(P_tW)+tr(M_0P_0),
\end{align}
%where, $\Lambda= [\lambda_1,\ldots,\lambda_D]^{\textsf{T}}$. 
Let us define two vectors $\gamma_t$ and $r_t$, as in the following:
\begin{align*}
\gamma_t&\!\triangleq\![c_{1,t},c_{2,t},\cdots, c_{D,t}]^\textsf{T},\\
r_t&\!\triangleq\![tr(\tilde P_tW), tr(\tilde P_tA^\textsf TWA), \cdots, tr(\tilde P_t{A^{D-1}}^\textsf TWA^{D-1})]^\textsf T.
\end{align*}
%$$\gamma_t\triangleq[c_{1,t},c_{2,t},\cdots, c_{D,t}]^\textsf{T}$$ and $$r_t\triangleq[tr(\tilde P_tW), tr(\tilde P_tA^\textsf TWA), \cdots, tr(\tilde P_t{A^{D-1}}^\textsf TWA^{D-1})]^\textsf T.$$ 
Since the term $\hat{x}_0^\textsf{T}P_0\hat{x}_0+\sum_{t=1}^Ttr(P_tW)+tr(M_0P_0)$ in (\ref{LP}), is independent of $\theta_{[0,T-1]}$, minimizing \eqref{LP} is equivalent to
\begin{align}\label{LP2}
\tilde{J}_0^*=\min_{\theta_{[0,T-1]}}\sum\nolimits_{t=0}^{T-1}\left[\gamma_t^\textsf Tr_t+\theta_t^\textsf{T}\Lambda \right].
\end{align}
After defining $(\gamma_k)_i$ to be the $i$-th component of the vector $\gamma_k$, the optimal strategy $\theta^*_{[0,T-1]}$ is the solution of the following mixed integer nonlinear programming (MINP):
\vspace{-.5mm}
\begin{align} \label{MINP}
\min_{\theta_{[0,T-1]}} ~~&\sum\nolimits_{k=0}^{T-1}\left[\theta_k^\textsf{T}\Lambda+\gamma_k^\textsf{T}r_k\right] \\\nonumber
\text{subject to} ~~& (\gamma_k)_i=\sum\nolimits_{j=i}^Db_{j,k}, \quad \sum\nolimits_{i=1}^D\theta_k^i\ge 1\\\nonumber
&b_{i,k}= \prod\nolimits_{d=1}^{i-1}\prod\nolimits_{j=1}^{d}(1-\theta^j_{k-d})( \vee_{l=1}^i\theta_{k-i}^l)\\\nonumber
&\sum\nolimits_{i=1}^{\tau_k}\!\!b_{i,k}=1, \quad \sum\nolimits_{i=k+2}^D \!b_{i,k}=0 \\\nonumber
&b_{i,k}\!\in \!\{0,1\},\theta_k^i \!\in \!\{0,1\}, ~ \forall k\!\in \![0,T\!-\!1], i\!\in \![1,D],
\end{align}
\begin{remark}
The optimal switching strategy $\theta^*_{[0,T-1]}$ is independent of the noise realizations and can be solved offline. This result is analogous to the conclusions of \cite{logothetis1999sensor}, wherein the optimal sensor schedule for a delay-free open-loop control system with linear Gaussian-disturbed sensors is shown to be independent of the Gaussian noise realizations.
\end{remark}

To significantly reduce the computational complexity of the MINP (\ref{MINP}), we show that the derived MINP can be equivalently re-casted as a mixed integer linear program\footnote{There are efficient algorithms and solvers for MILPs, whereas often the LP relaxation of the MILP results in a solution close to the optimal.}
 (MILP), by exploiting certain structure of the specific network setting. This will significantly reduce the computational burden as well as might provide a way to further relaxed it to a linear programming (LP) problem.  
For this, by replacing $\sum\nolimits_{i=1}^D\theta_k^i\ge 1$ with $\sum\nolimits_{i=1}^D\theta_k^i= 1$ enables us to replace $\vee_{l=1}^i\theta_{k-i}^l$ in (\ref{MINP}) by $\sum_{l=1}^i\theta_{k-i}^l$. Thus,
\vspace{-2mm}
\begin{align} \label{MILNP}
\min_{\theta_{[0,T-1]}} ~~&\sum\nolimits_{t=0}^{T-1}\left[\theta_t^\textsf{T}\Lambda+\gamma_t^\textsf{T}r_t\right] \\\nonumber
\text{subject to} ~~& (\gamma_t)_i=\sum\nolimits_{j=i}^Db_{j,t}, \quad \sum\nolimits_{i=1}^D\theta_t^i= 1\\\nonumber
&b_{i,t}= \prod\nolimits_{d=1}^{i-1}\prod\nolimits_{j=1}^{d}(1-\theta^j_{t-d})( \sum_{l=1}^i\theta_{t-i}^l)\\\nonumber
&\sum\nolimits_{i=1}^{\tau_k}b_{i,t}=1, \quad \sum\nolimits_{i=t+2}^D \!b_{i,t}=0 \\\nonumber
&b_{i,t}\!\in \!\{0,1\},\theta_t^i \!\in \!\{0,1\}, ~ \forall t\!\in \![0,T\!-\!1], i\!\in \![1,D].
\vspace{-2mm}
\end{align}
Clearly, due to the conversion of an inequality constraint to an equality constraint, every feasible solution of \eqref{MILNP} is a feasible solution for \eqref{MINP}, and moreover the optimal value for \eqref{MILNP} is no less than that of \eqref{MINP}. Therefore, we only need to show that an optimal solution for \eqref{MINP} is a feasible solution for \eqref{MILNP}. To show this, we first claim that every $\theta$ which is feasible for \eqref{MINP} but not for \eqref{MILNP} (i.e. $\sum\nolimits_{i=1}^D\theta_k^i> 1$ for some $k$), there exists a $\tilde \theta$ which achieves a strictly lesser cost than $\theta$.
%To derive the equivalent MILP, let us first start by taking a feasible solution $\theta_{[0,T-1]}$ of the MINP such that for some $k \!\in \![0,T\!-\!1]$,  $\sum_{i=1}^D\theta_k^i\!>\!1$\footnote{If no such $k$ exists, then according to the constraint \eqref{const1}, $\sum_{i=1}^D\theta_k^i=1$ for all $k$ and we can replace $\vee_{l=1}^i\theta_{t-i}^l$ by $\sum_{l=1}^i\theta_{t-i}^l$ in the MINP.}. 
Let $1\!\le \!i_1\!<\!i_2\!<\!\cdots\!<\!i_m \!\le \!D$ be the indices such that $\theta_k^{i_n}\!=\!1$. Now we construct a new $\tilde{\theta}_{k}$ such that $\tilde \theta_k^{i_1}\!=1$, and  $\tilde \theta^{j}_k\!=0$, for all $j\ne i_1$. Thus, $\sum_{i=1}^D\tilde \theta_k^i\!=\!1$, whereas, $\sum_{i=1}^D\theta_k^i\!>\!1$. This is done for each $k$ such that $\sum\nolimits_{i=1}^D\theta_k^i> 1$.
It can be verified that the cost $\sum\nolimits_{t=0}^{T-1}\gamma_t^{\T}r_t$ remains the same while using $\theta_{[0,T-1]}$ or $\tilde{\theta}_{[0,T-1]}$; whereas $\sum\nolimits_{t=0}^{T-1}\theta_t^{\T}\Lambda > \sum\nolimits_{t=0}^{T-1}\tilde{\theta}_t^{\T}\Lambda$. Thus the optimal solution of \eqref{MINP} must be the optimal solution of \eqref{MILNP}. 
%$\tilde \theta_{[0,T-1]}$ is a feasible solution of the MINP in \eqref{MINP}. Moreover, one can verify that by employing $\tilde \theta_{[0,T-1]}$ instead of $\theta_{[0,T-1]}$ the communication cost is reduced by $\sum_{n=2}^m\lambda_{i_n}$. Hence, for all $k_1$ such that $\sum_{i=1}^D\tilde{\theta}_{k_1}^i\!>\!1$, we can repeat the same process to achieve $\sum_{i=1}^D\tilde{\theta}_{k}^i\!=\!1$, for all $k$. 
%Thus, it is sufficient to search only for optimal $\theta_{[0,T-1]}$ in the set $\mathfrak{T}\!=\!\{ \theta_{[0,T-1]}| \theta_j^i \!\in \!\{0,1\}, \sum_{i=1}^D\theta_j^i\!=\!1, \forall j\!\in\![1,\cdots,T\!-\!1], i\!=\!\{1,\cdots, D\}\}$. We can incorporate this as an extra constraint to the MINP in \ref{MINP}, leading to $b_{i,t}\!=\! \prod_{d=1}^{i-1}\prod_{j=1}^{d}(1-\theta^j_{k-d})( \sum_{l=1}^i\theta_{t-i}^l)$.
% However, \eqref{MILNP} is still an MINP due to nonlinearity of the constraint on $b_{i,t}$. 
Relaxing the equality constraint of $b_{i,t}$ as $b_{i,t} \!\le\!( \sum_{l=1}^i\theta_{t-i}^l)$ results in the following MILP which is equivalent to \eqref{MINP}:
\begin{align} \label{MILP}
\min_{\theta_{[0,T-1]}} ~~&\sum\nolimits_{t=0}^{T-1}\left[\theta_t^\textsf{T}\Lambda+\gamma_t^\textsf{T}r_t\right] \\\nonumber
\text{subject to} ~~& (\gamma_t)_i=\sum\nolimits_{j=i}^Db_{j,t}, \quad b_{i,t} \le  \sum\nolimits_{l=1}^i\theta_{t-i}^l\\\nonumber
&\sum\nolimits_{i=1}^D\!\theta_t^i= 1, \;\sum\nolimits_{i=1}^{\tau_k}\!b_{i,t}=1, \;\sum\nolimits_{i=t+2}^D\!b_{i,t}=0\\\nonumber
&b_{i,t}\!\in \!\{0,1\},\theta_t^i \!\in \!\{0,1\}, ~ \forall t\!\in \![0,T\!-\!1], i\!\in \![1,D].
\end{align}

Problem \eqref{MILP} is a relaxed version of (\ref{MILNP}), therefore, any optimal solution of \eqref{MILP} is also an optimal solution of \eqref{MILNP} if it is a feasible solution for \eqref{MILNP}. At this point, it is trivial to verify that the optimal solution of \eqref{MILP} is a feasible (and hence optimal) solution for \eqref{MILNP}, and hence optimal for \eqref{MINP}.

\vspace{-.7mm}

\subsection{Communication Cost as a Constraint}\label{subsec:budget}

So far, we have considered the cost function of the form
\begin{align*}
J=\min\nolimits_{u,\theta}\E[J_{LQG}+J_{Comm}],
\end{align*}
where, %$J_{LQG}$ is the quadratic cost over state and control, and 
$J_{Comm}$ is the communication cost. There are equivalent formulations of the this problem depending on the specific NCSs setup, e.g., constraint optimization problem:  
\begin{align*}
 \min_{u_{[k,T-1]},\;\theta_{[0,T-1]}} &\;\E\!\left[\sum\nolimits_{t=0}^{T-1}\!x_t^{\textsf{T}}Q_1x_t\!+\!u_t^{\textsf{T}}Ru_t\!+\!x_T^{\textsf{T}}Q_2x_T\!\big|\mathcal{I}_k\right], \\
\text{s.t.} &\;\E\left[\sum\nolimits_{t=0}^{T-1}\theta^{\T}_t \Lambda\right] \le b,
\end{align*}
 where  $b\in \mathbb{R}^+$  is the budget; or, a bi-objective problem:
\begin{align}\label{bi_obj_opt}
\min_{u_{[k,T-1]},\;\theta_{[0,T-1]}} &\left\{f_1(u,\theta), f_2(u,\theta) \right\},
\end{align}
with $f_1=\E\!\left[\sum_{t=0}^{T-1}\!x_t^{\textsf{T}}Q_1x_t\!+\!u_t^{\textsf{T}}Ru_t\!+\!x_T^{\textsf{T}}Q_2x_T\!\right] $, and $f_2=\E\left[\sum_{t=0}^{T-1}\theta_t^{\T} \Lambda\right]$. 
%All these problems are similar in some sense and Problem \ref{Prob:bi-object} is the most general of them. 
The solution of the bi-objective problem is characterized by Pareto frontier. Looking at Pareto curve for the section $f_2\!\le \!b$, one obtains the solution of the constrained budget problem. Moreover, solving the constrained budget problem for all $b \ge 0$, the Pareto frontier for the bi-objective problem is obtained. 
\begin{lemma} \label{L:pareto}
Consider the multi-objective problem
\begin{align}\label{E:mult}
\min_{s}\{f_1(s),f_2(s),\cdots,f_m(s)\},
\end{align}
and the equivalent weighted cost problem:
\begin{align}\label{E:mult1}
\min_s&\sum\nolimits_{i=1}^m\alpha_if_i(s), \quad \text{s.t.} \;\sum\nolimits_{i=1}^m\alpha_i=1;~~ \alpha_i\ge 0,
\end{align}
then $s^*$ is a Pareto point  for \eqref{E:mult} if and only if $s^*$ is the solution of \eqref{E:mult1} for some $\{\alpha_i\}$.
\end{lemma}
The Pareto frontier for (\ref{bi_obj_opt}) can be constructed by optimizing the single objective function 
\begin{align} \label{E:weightedProblem}
\!\E\!\left[\sum_{t=0}^{T-1}\!\alpha\!\left[x_t^{\textsf{T}}Q_1x_t\!+\!u_t^{\textsf{T}}Ru_t\!+\!x_T^{\textsf{T}}Q_2x_T\right]\!\!+(1\!-\!\alpha)\theta^{\T}_t \Lambda\right]\!,
\end{align}
for all $\alpha\in [0,1]$. Note that \eqref{E:weightedProblem} is equivalent to \eqref{joint_cost} which can be solved following the discussion presented here.
%\begin{remark}
%Stability analysis is not addressed in this article due to space limitations. However, assuming $[x_k,e_k]$ to be the NCS state, NCS stability for the asymptotic optimal control and switching policies, i.e. $\lim\limits_{k\rightarrow\infty}(u^{\ast}_k,\theta^{\ast}_k)$, can be shown in mean-square sense, under controllability of $(A,B)$.
%\end{remark}

%% file: Simulation.tex
 	%%%%%%%%%%%%%%%%%%%%%%%%%%%%%%%%%%%%%%%%%%%%%%%%%%%%%%%%%%%%%%%%%%%%%%%%%%%%%%%%
\section{Simulation Results}\label{sim_results}
\subsection{Example 1: Unstable Dynamics}
Consider an NCS with unstable dynamics as:
\begin{align*}
x_{t+1}=\begin{bmatrix}
1.01 & 0\\0 &1
\end{bmatrix}x_t
+\begin{bmatrix}
0.1 &0\\0 & 0.15
\end{bmatrix}u_t+\sqrt{1.5}w_t
\vspace{-2mm}
\end{align*}
where, $w_t, x_0\sim \mathcal{N}(0,\mathbb{I}_2)$. %and $x_0\sim \mathcal{N}(0,\mathbb{I}_2)$. 
The horizon $T$ is set to be $100$. There are 5 links with delays ranging from $1$ to $5$ time-steps and the corresponding prices are  $[20,13, 8, 2, 1]$. The optimal utilization of the links is shown in Fig.~\ref{plot1}.
\begin{figure}
\includegraphics[width=0.45 \textwidth, height= 3 cm]{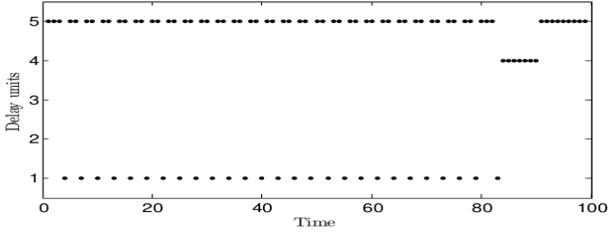}
\vspace{-0mm}\caption{Optimal utilization of the links} \label{plot1}
\vspace{-0mm}
\end{figure}
\begin{figure}
\begin{subfigure}[t]{0.241 \textwidth}
\centering
\includegraphics[width=\textwidth]{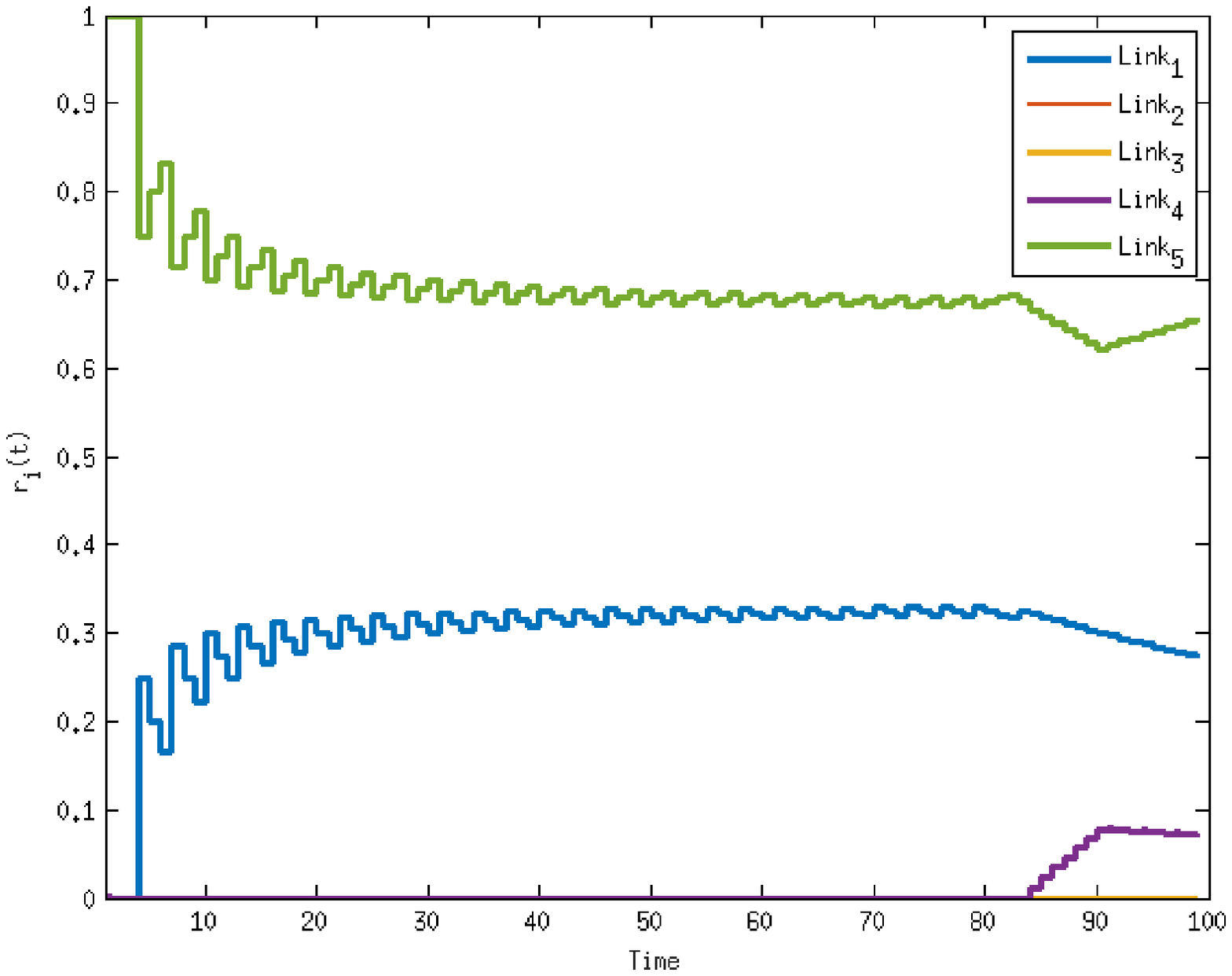}
\vspace{-0mm}\subcaption{Utilizations of different links over time: $\rho_i(t)$} \label{plot2}
\end{subfigure}
\begin{subfigure} [t]{0.24 \textwidth}
\centering
%\psfrag{Control cost}[c][c]{\tiny \textsf{Communication cost}}
\includegraphics[width=\textwidth]{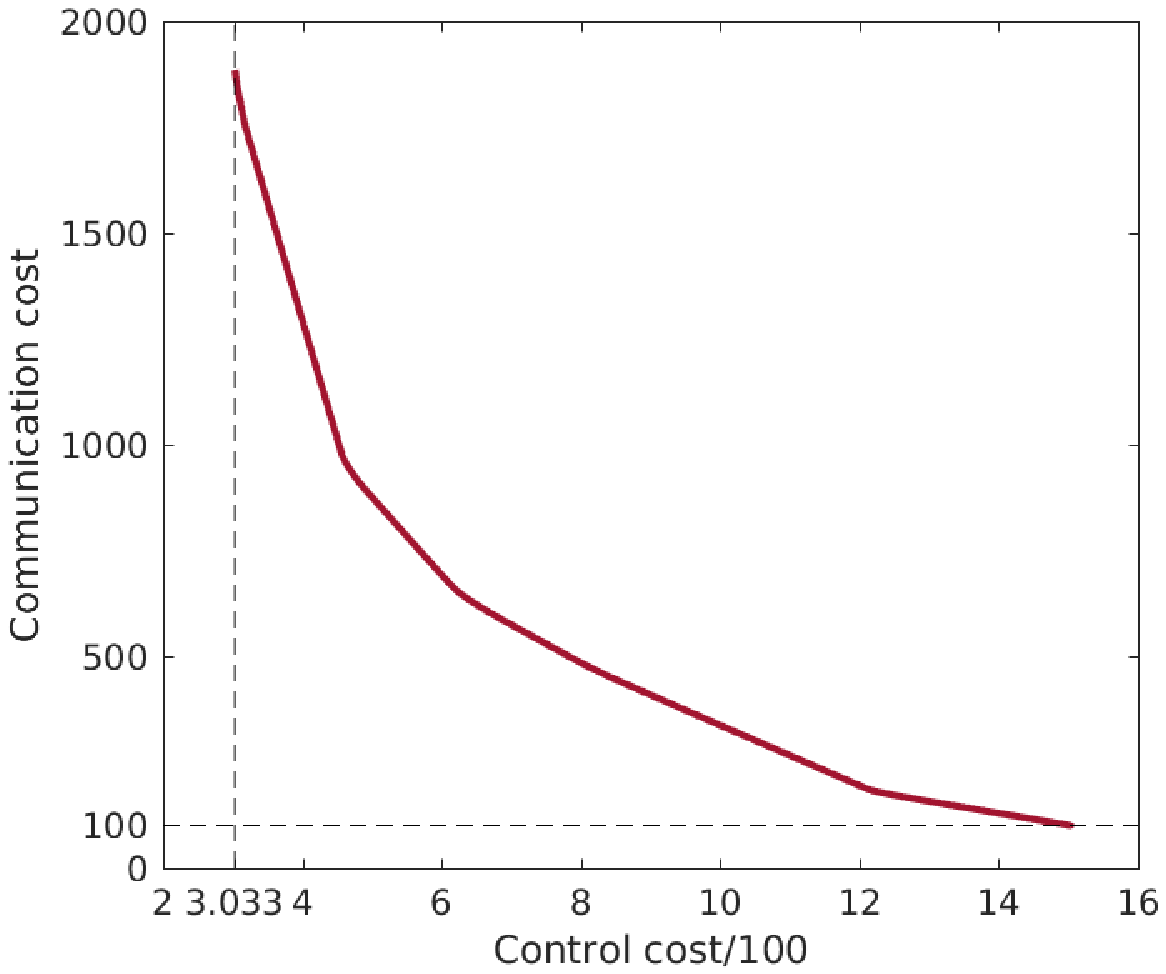}
\vspace{-0mm}\subcaption{Pareto front of the bi-objective problem with $\Lambda=[20 ~13 ~8 ~2~1]$. %The lowest LQG cost achievable is 303.3, whereas the lowest communication cost is 100.
} \label{plot3}
\end{subfigure}
%\vspace{-5mm}
\end{figure}
For this choice of the parameters the network mainly uses the fastest (link $1$) and the slowest (link $5$) links. Only for few instances, the system utilizes the link $4$ and the rest of the links are not used. Thus, we note that the measurements sent by the Link 5 is never used in estimation except towards the end. Thus, the system can remain open-loop for most of the time. 

To assure our simulation setup accuracy, we set $\lambda_i\!=\!0$ for all links, and we observe that only the fastest link is selected. Similarly, setting $\lambda_i \!\gg \!1$, the system selects the slowest link, as the communication cost is exorbitantly high compared to the LQG cost. Similar profile is observed for all~$\Lambda$, when disturbance is removed, and system becomes deterministic, so the only observation required is the initial state, and no need to send any measurement at all. However, the constraint $\sum_{i=1}^D\theta_k^i\!\ge \!1$, forces the system to select the slowest link.

Let $\rho_i(t)$ be defined as $$\rho_i(t)=\frac{\text{total utilization number of link} ~i}{t}.$$ 
%\begin{align*}
%\rho_i(t)=\frac{\text{total utilization number of link} ~i}{t}
%\end{align*}
In Fig. \ref{plot2}, %$\rho_i(t)$ is plotted. %with different parameters for the simulations, 
we observe that mostly two of the links (fastest and slowest) are utilized, while the rest are hardly used. This behavior is linked with the structure of the MILP \eqref{MILP}, and studying it is beyond the scope of this article. However, this raises an interesting question for multiple systems scenario: How could the links be distributed among sub-systems so that the link utilization is fair? Also, we observe that $\rho_i(t)$ is very sensitive to the variations of $\lambda_i$. The design of prices $\Lambda$, as a time-varying or state-dependent variable, to achieve a desired utilization profile, is the subject of our future study. %based on a utility function for the network-manager. %This is the subject of future research on this topic.

In Fig. \ref{plot3} we show the Pareto frontier of the bi-objective problem defined in (\ref{bi_obj_opt}). We notice that the minimum LQG cost (with fastest link being always selected) achievable for this set of parameters is 303.3 and the maximum LQG cost (with cheapest link being always selected) is 1503. The minimum communication cost is 100 (since cheapest link cost =1) which is associated with the maximum LQG cost.

\subsection{Example 2: Stable dynamics}
\begin{align*}
\dot x=\begin{bmatrix}
0.5 & 0.05 \\0.5 & 0.9
\end{bmatrix}x+\begin{bmatrix}
0.1&0.01 \\ 0.05 & 0.15
\end{bmatrix}u+\sqrt{1.5}w_t
\end{align*}
In this setting, we similarly choose a network with $5$ possible delay links (delays are 1,2,3,4,5 time steps.) with the costs: $[10,~ 8,~ 2.5,~ 1.5,~ 1]$, such that lower delays are assigned with higher costs. The network utilization is shown in Figure \ref{F:1} for this system, which follows the similar pattern as in the previous example.

\begin{figure}
\begin{subfigure}[t]{0.241 \textwidth}
\includegraphics[width=\textwidth]{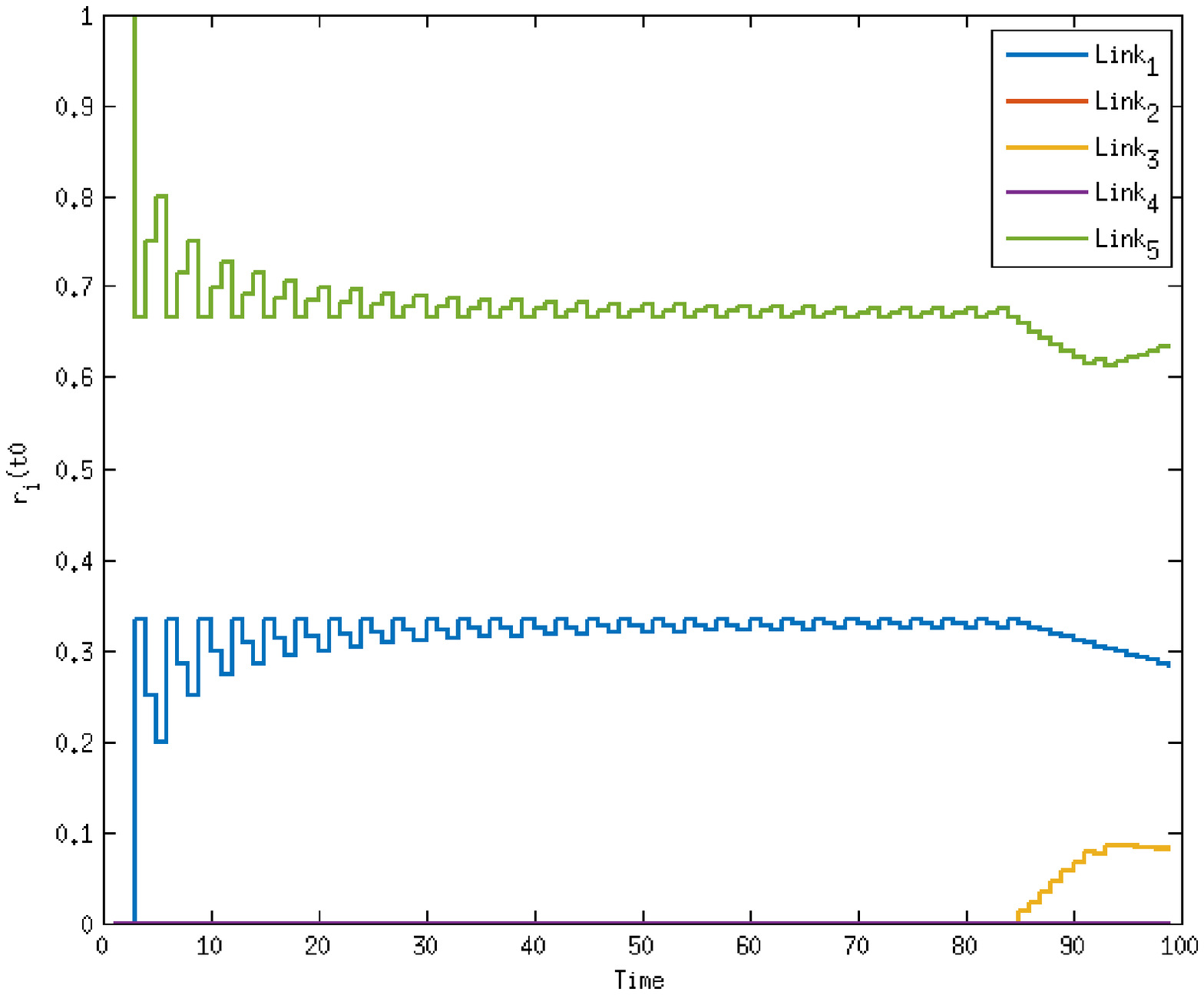} 
\end{subfigure}
\begin{subfigure}[t]{0.24 \textwidth}
\includegraphics[width=\textwidth]{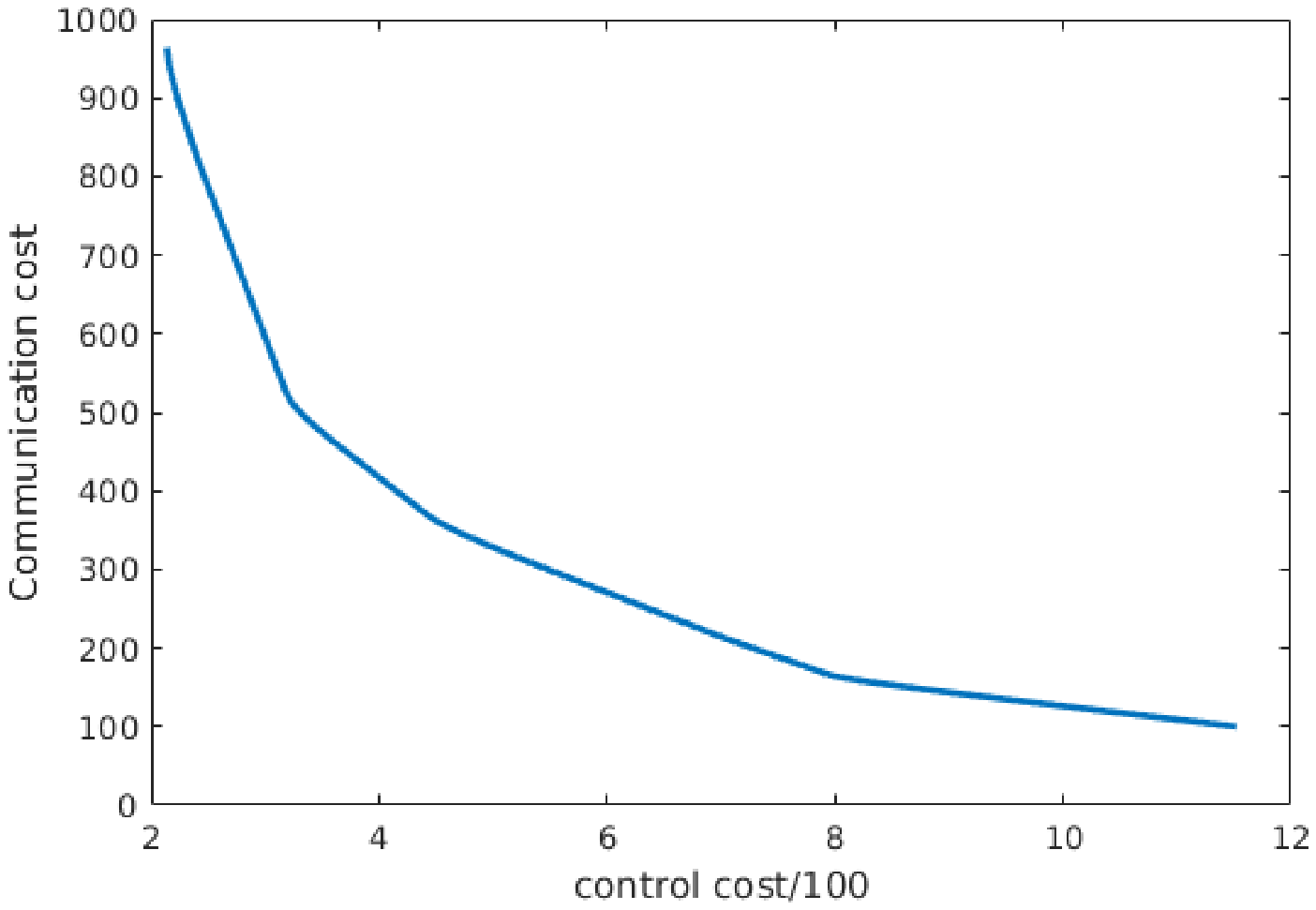}
\end{subfigure}
\caption{First: network utilizatin over time $t$, Second: Pareto curve of the bi-objective problem}\label{F:1}
\end{figure}

%\begin{figure} 
%\includegraphics[width=0.25\textwidth]{pareto.eps}
%\caption{Pareto Front of the bi-objective problem with costs $\Lambda=[20 ~13 ~8 ~2~1]$. The lowest LQG cost achievable is 303.3, whereas the lowest communication cost is 100.} \label{plot3}
%\end{figure}

%% file: conclusions.tex
\section{Conclusion}\label{conclusion}
In this article we address the problem of joint optimal LQG control and delay switching strategy in an NCS with a single stochastic LTI system. Assuming that the network utilization incurs cost, i.e. transmission with shorter delay is more costly, we derive the optimal delay switching profile. The overall cost function consisting of the LQG cost plus communication cost, is shown to be decomposable in expectation assuming \textit{apriori} known prices. Having the separation property, the optimal laws can be computed offline as the solutions of an algebraic Riccati equation for the optimal control law, and a MILP, for the optimal switching profile. 

%% file: Discussion.tex
\section{Discussion}\label{discussion}
\subsection{Stability Analysis}
First, it is worth mentioning that in this work we study the finite horizon optimal design of control and transmission policies, and due to the linearity of the system dynamics, it is ensured that the state of the system remains bounded in expectation over any finite horizon time duration, if the cost of communication is finite. Therefore, talking about stability, which is an asymptotic system property, we look into the infinite time horizon. To do this, as we will discuss in the following (see Lemma \ref{lemm:convergenceRiccati}), we consider the designed strategies in their limit ($k\rightarrow \infty$) and show that under finite prices for communication, the system is asymptotically stable in mean-square sense.

We study stability of the described control system in infinite horizon under the steady-state optimal control and transmission policies, i.e. the limit of the strategies $(u^\ast_k,\theta_k^\ast)$, derived in the manuscript. It is straightforward to express the evolution of the system state $x_k$, as follows:
\begin{align}\label{eq:sys_state}
x_{k+1}&=(A-BL_k)x_k+BL_ke_k+w_k,
\end{align}
where, the estimation error $e_k$ is defined in section III-A, and $L_k=\left(R+B^\textsf{T}P_{k+1}B\right)^{-1}B^\textsf{T}P_{k+1}A$. Due to the existence of exogenous stochastic disturbance , we employ the concept of Lyapunov mean-square stability (LMSS), to investigate stability properties of the closed-loop system (\ref{eq:sys_state}). Let us first define the notion of LMSS, as follows.
\begin{definition}[\cite{c1}]\label{def:LMSS}\normalfont
An LTI system with state vector $X_k$ is said to be Lyapunov mean-square stable (LMSS) if given $\varepsilon \!>\!0$, there exists $\rho(\varepsilon)$ such that $\|X_0\|_2\!<\!\rho$ implies 
\begin{equation}\label{LMSS}
\sup_{k\geq 0}\textsf{E}\left[\|X_k\|_2^2\right]\leq \varepsilon.
\end{equation}
\end{definition}

From (9), one can see that the evolution of $e_k$ is independent of the system state $x_k$, and is only dependent on the noise realizations, the system matrix $A$, and the transmission policy $\theta_{[k-D,k-1]}$. Therefore, if it is shown that the optimal control gain $L_k$ ensures asymptotic stability of the closed-loop system $x^c_{k+1}=(A-BL_k)x_k^c$, where $x_k^c$ is the state of the system (\ref{eq:sys_state}) in ideal case, i.e. assuming that no transmission delay exists, then LMSS is reduced to showing that the quadratic norm of the error state $e_k$ is, in the limit, bounded in expectation. Before stating the main stability result, we revisit the following definition and Lemma:

\begin{definition}[\cite{c2}]\label{def:UASL}\normalfont
A dynamical system $x_{k+1}\!=\!f(x_k,k)$, $x_0=x(0)$, is said to be \textit{uniformly asymptotically stable in large} (UASL) with respect to $x^\ast$ if the followings hold
\begin{itemize}
\item[i.] given $\varepsilon >0$, there exists $\delta >0$ such that $\|x^\ast-x_0\|\leq \delta$ implies that $\|x_k-x^\ast\|\leq \varepsilon$, for any $k>0$,
\item[ii.] given $\delta >0$, there exists $\varepsilon >0$ such that $\|x^\ast-x_0\|\leq \delta$ implies that $\|x_k-x^\ast\|\leq \varepsilon$, for any $k>0$.
\end{itemize}
\end{definition}

\begin{lemma}[\cite{c2}]\label{lemm:convergenceRiccati}\normalfont
Let $P_k$ be the solution of the following discrete Riccati operator equation
\begin{align*}
&P_k=Q_1+A^\textsf{T}\!\left(\!P_{k+1}-P_{k+1}B\left(R+B^\textsf{T}P_{k+1}B\right)^{-1}\!B^\textsf{T}P_{k+1}\!\right)\!A,
\end{align*}
where, $Q_1=Q_1^{\frac{1}{2}^\textsf{T}}Q_1^{\frac{1}{2}}$. If $(A,B)$ is stabilizable, and $(A,Q_1^{\frac{1}{2}})$ detectable, then $P_k$ converges, in the operator norm, as $k\rightarrow \infty$, to a positive operator $P$ that is the unique positive solution to the associated algebraic Riccati equation
\begin{align*}
P=Q_1+A^\textsf{T}\!\left(\!P-PB\left(R+B^\textsf{T}PB\right)^{-1}\!B^\textsf{T}P\!\right)\!A.
\end{align*}
In addition, the control and state generated by 
\begin{align*}
x_{k+1}&=Ax_k+Bu_k, \quad x_0=x(0),\\
u_k&=-\left(R+B^\textsf{T}PB\right)^{-1}\!B^\textsf{T}PA x_k,
\end{align*}
is uniformly asymptotically stable in large (UASL) with respect to the origin $x^\ast=0$.
\end{lemma}

\begin{theorem}\label{thm:stability}\normalfont
Consider an LTI system described by the discrete time dynamics (1) under the optimal state-feedback control $u^\ast_k$, and the optimal transmission control $\theta^\ast_k$ derived according to (5). Under the controllability and stabilizability assumptions, and also assuming that the constraint (2) holds at every time-step $k\in\{1,2,\ldots\}$, the closed-loop system $x_{k+1}=(A-BL_k)x_k+BL_ke_k+w_k$ is Lyapunov mean-square stable.
\end{theorem}

\begin{proof}\normalfont
Since the pairs $(A,B)$and $(A,Q_1^{\frac{1}{2}})$ are controllable and detectable, respectively, the closed-loop system $x^c_{k+1}=(A-BL_k)x_k^c$ is ensured to be UASL, assuming that no transmission delay exists, i.e. $e_k=0$. This conclusion follows from Lemma \ref{lemm:convergenceRiccati}. Therefore, the closed-loop matrix $(A-BL_k)$ is stable as $k\rightarrow \infty$ and $\lim_{k\rightarrow \infty}|x_k^c|\rightarrow 0$. Moreover, due to linearity of the system, it is ensured that there exists $\varepsilon_c>0$ such that $\textsf{E}\left[\|x^{c^{\textsf{T}}}_{k}x^c_k\|_2^2\right]\leq \varepsilon_c$ at any time $k$. Considering the transmission delays, the term $e_k$ appears in the dynamics. However, as $e_k$ is independent of $x_k$, it is sufficient to show the quadratic norm of error state $e_k$ is bounded in expectation for any $k\leq 0$. Hence, according to the Definition \ref{def:LMSS}, it must be ensured that, given $0<\bar{\varepsilon} <\varepsilon$, there exists $\bar{\rho}(\bar{\varepsilon})$ such that $\|e_0\|_2\!<\!\bar{\rho}$ implies
 \begin{equation*}
\sup_{k\geq 0}\textsf{E}\left[\|e_k\|_2^2\right]\leq \bar{\varepsilon}.
\end{equation*}
We consider the worst case evolution of error state over time, which happens when the transmission is always performed with the maximum delay, i.e. $D$, and consequently, the system is not updated for maximum of $D-1$ time-steps. We then evaluate the dynamics of error over any interval of length $D-1$ over which no state information has been received. Assume that at a time-step $k$ the controller has received one state information, and then the next update happens at time $k+D$. 
%This is ensured if the transmission policy over the time period $[k-D,k+D-1]$ is as follows: 
%$$\theta_{k-1}^1=\sum_{i=1}^2\theta_{k-2}^i=\ldots=\sum_{i=1}^{D-1} \theta_{k-D+1}^i=\sum_{i=1}^D\theta_{k-D}^i=1,$$
%and
%$$\theta_{k}^D=\sum_{i=D-1}^D\theta_{k+1}^i=\ldots=\sum_{i=2}^{D} \theta_{k+D-2}^i=\sum_{i=1}^D\theta_{k+D-1}^i=1.$$
Therefore, over the interval $\left(k,k+D-1\right]$, the controller receives no state information. From (9), we know
\begin{equation*}
e_{k+D-1}=\sum_{j=1}^D A^{j-1}w_{k+D-1-j}.
\end{equation*}
It then follows that
\begin{align*}
\textsf{E}\left[\|e_{k+D-1}\|_2^2\right]&=\textsf{E}\left[\|\sum_{j=1}^D A^{j-1}w_{k+D-1-j}\|_2^2\right]\\
&=\sum_{j=1}^D \E\left[\|A^{j-1}w_{k+D-1-j}\|_2^2\right]\\&\leq  \sum_{j=1}^D \|A^{j-1}\|_2^2 \Tr(W),
\end{align*}
where, the second equality follows due to statistical independence of the noise realizations, and the inequality follows from the multiplicativity property of matrix norms. Note that, as the time interval $\left(k,k+D-1\right]$ is generic, this is the possible maximum error norm the system is expected to have at any time-step, due to not having an update for the last $D-1$ time-steps. This ensures that the inequality is valid for any other transmission sequence $\theta_k$ for $k\in\{1,2,\ldots\}$, optimal or non-optimal. Finally, boundedness of the error state $e_k$ and system state $x_k^c$ in expectation ensures that $x_k$ is also bounded in expectation at any time $k\in\{1,2,\ldots\}$, which satisfies LMSS condition in \textit{Definition} 1, and the proof is then complete. 
\end{proof}

%% file: appendix.tex
%%%%%%%%%%%%%%%%%%%%%%%%%%%%%%%%%%%%%%%%%%%%%%%%%%%%%%%%%%%%%%%%%%%%%%%%%%%%%%%%
\section{Appendix}\label{Appendix}
\subsection{Proof of Theorem \ref{thm:separation}} \label{app}
The LQG optimal value function at time $k+1$ is 
\begin{align}\nonumber
&V^\ast_{k+1}\!=\min_{u_{[k+1,T-1]}}\!\!\E\!\left[\sum_{t=k+1}^{T-1}\!\!x_t^{\textsf{T}}Q_1x_t\!+\!u_t^{\textsf{T}}Ru_t\!+\!x_T^{\textsf{T}}Q_2x_T|\mathcal{\bar{I}}_{k+1}\!\right]\!.
\end{align} 
Knowing that $\mathcal{\bar{I}}_k\subset \mathcal{\bar{I}}_{k+1}$, the law of total expectation yields
\begin{align}\nonumber
&\!\E\!\left[\!V^\ast_{k+1}|\mathcal{\bar{I}}_k\right]\!=\!\!\min_{u_{[k+1,T\!-\!1]}}\!\!\E\!\!\left[\!\sum_{\;t=k+1}^{T-1}\!\!\!\!x_t^{\textsf{T}}Q_1x_t\!+\!u_t^{\textsf{T}}Ru_t\!+\!x_T^{\textsf{T}}Q_2x_T|\mathcal{\bar{I}}_{k}\!\right]
\end{align} 
Therefore, it is straightforward to re-write $V^\ast_{k}$ as follows:
\begin{align}\label{value_function1}
V^\ast_k&=
%\min_{u_{k}}\E\left[x_k^{\textsf{T}}Q_1x_k+u_k^{\textsf{T}}Ru_k|\mathcal{\bar{I}}_k\right]\\\nonumber
%&+\min_{u_{[k+1,T-1]}}\E\!\left[\sum_{t=k+1}^{T-1}x_t^{\textsf{T}}Q_1x_t+u_t^{\textsf{T}}Ru_t+x_T^{\textsf{T}}Q_2x_T|\mathcal{\bar{I}}_k\right]\\
\min_{u_{[k,T-1]}}\E[x_k^{\textsf{T}}Q_1x_k+u_k^{\textsf{T}}Ru_k+V^\ast_{k+1}|\mathcal{\bar{I}}_k].
\end{align}

Assume that $V^\ast_k$ can be expressed as follows:
\begin{equation}\label{value_function_re}
V^\ast_k=\E[x_k|\mathcal{\bar{I}}_k]^{\textsf{T}}P_k \E[x_k|\mathcal{\bar{I}}_k]+\pi_k \triangleq \hat{x}_k^{\textsf{T}}P_k \hat{x}_k+\pi_k,
\end{equation}
where, $\pi_k$ will be derived later as a term independent of the control $u_k$. Having (\ref{value_function_re}) assumed, (\ref{value_function1}) can be re-written~as
\begin{align}\label{value_function1_re}
&V^\ast_k=\\\nonumber
&\min_{u_{k}}\E\!\left[x_k^{\textsf{T}}Q_1x_k+u_k^{\textsf{T}}Ru_k+\hat{x}_{k+1}^{\textsf{T}}P_{k+1} \hat{x}_{k+1}+\pi_{k+1}|\mathcal{\bar{I}}_k\right]\!.
\end{align}
We define the apriori state estimate $\hat{x}_k^-\triangleq\E\left[x_k|\mathcal{\bar{I}}_{k-1}\right]=A\hat{x}_k+Bu_k$. Due to the fact that 
\begin{align*}
\!\E\!\!\left[\!\hat{x}_{k+1}^{\textsf{T}}\!P_{k+1}\hat{x}_{k+1}^-|\mathcal{\bar{I}}_k\!\right]\!\!=\!\E\!\!\left[\!\hat{x}_{k+1}^{-^{\textsf{T}}}\!P_{k+1}\hat{x}_{k+1}|\mathcal{\bar{I}}_k\!\right]\!\!=\!\hat{x}_{k+1}^{-^{\textsf{T}}}\!P_{k+1}\hat{x}_{k+1}^-
%\E\left[\hat{x}_{k+1}^{-^{\textsf{T}}}P_{k+1}\hat{x}_{k+1}|\mathcal{\bar{I}}_k\right]&=\hat{x}_{k+1}^{-^{\textsf{T}}}P_{k+1}\hat{x}_{k+1}^-.
\end{align*}
then, (\ref{value_function1_re}) can be written as in the following: %a function of $\hat{x}_{k+1}^-$, as
\begin{align}\nonumber
V^\ast_k&=\min_{u_{k}}\E[x_k^{\textsf{T}}Q_1x_k+u_k^{\textsf{T}}Ru_k|\mathcal{\bar{I}}_k]\\\nonumber
&+\min_{u_{k}}\E[(A\hat{x}_k+Bu_k)^{\textsf{T}}P_{k+1}\left(A\hat{x}_k+Bu_k\right)|\mathcal{\bar{I}}_k]\\\label{value_function1_final}
&+\xi_{k+1}^{\textsf{T}}P_{k+1}\xi_{k+1}+\pi_{k+1},
\end{align}
where, $\xi_{k+1}\!\triangleq  \!\hat{x}_{k+1}-\hat{x}_{k+1}^-$.
%\begin{align}\label{value_function1_re2}
%V^\ast_k&=\min_{u_{k}}\E\!\left[x_k^{\textsf{T}}Q_1x_k+u_k^{\textsf{T}}Ru_k|\mathcal{\bar{I}}_k\right]\\\nonumber
%&+\min_{u_{k}}\E\!\left[\!\left(\hat{x}_{k+1}\!-\!\hat{x}_{k+1}^-\right)^{\textsf{T}}\!P_{k+1} \!\left(\hat{x}_{k+1}\!-\!\hat{x}_{k+1}^-\right)|\mathcal{\bar{I}}_k\right]\\\nonumber
%&+\min_{u_{k}}\E\!\left[\hat{x}_{k+1}^{-^{\textsf{T}}}P_{k+1}\hat{x}_{k+1}^-\!+\!\pi_{k+1}|\mathcal{\bar{I}}_k\right]\!
%\end{align}
%\begin{align}\label{value_function1_final}
%V^\ast_k&=\min_{u_{k}}\E\!\left[x_k^{\textsf{T}}Q_1x_k+u_k^{\textsf{T}}Ru_k|\mathcal{\bar{I}}_k\right]\\\nonumber
%&+\min_{u_{k}}\E\!\left[\left(A\hat{x}_k+Bu_k\right)^{\textsf{T}}P_{k+1}\left(A\hat{x}_k+Bu_k\right)|\mathcal{\bar{I}}_k\right]\\\nonumber
%&+\xi_{k+1}^{\textsf{T}}P_{k+1}\xi_{k+1}+\pi_{k+1}.
%\end{align}
It is then simple to derive the optimal control $u^\ast_k$, minimizing (\ref{value_function1_final}), which is of the form
 \begin{equation*}
u^\ast_k=-(R+B^\textsf{T}P_{k+1}B)^{-1}B^\textsf{T}P_{k+1}A \hat{x}_k.
\end{equation*}
Plugging the optimal control $u^\ast_k$ in (\ref{value_function1_final}), together with replacing $x_k$ with its equivalent expression $e_k+\hat{x}_k$, result in\vspace{-2mm}
\begin{align}\label{value_function1_optimal}
V^\ast_k&=
%\E\!\left[\left(e_k+\hat{x}_k\right)^{\textsf{T}}Q_1\left(e_k+\hat{x}_k\right)|\mathcal{\bar{I}}_k\right]\\\nonumber
%&+\E\!\left[\hat{x}_k^{\textsf{T}}\left(\tilde{B}^{\textsf{T}}R\tilde{B}+\tilde{A}^{\textsf{T}}P_{k+1}\tilde{A}\right)\hat{x}_k|\mathcal{\bar{I}}_k\right]\\\nonumber
%&+\xi_{k+1}^{\textsf{T}}P_{k+1}\xi_{k+1}+\pi_{k+1}\\
\E[\hat{x}_k^{\textsf{T}}(\tilde{B}^{\textsf{T}}R\tilde{B}+\tilde{A}^{\textsf{T}}P_{k+1}\tilde{A}+Q_1)\hat{x}_k|\mathcal{\bar{I}}_k]\\\nonumber
&+\E[e_k^{\textsf{T}}Q_1e_k+\xi_{k+1}^{\textsf{T}}P_{k+1}\xi_{k+1}+\pi_{k+1}|\mathcal{\bar{I}}_k],
\end{align}
where, $\tilde{B}_k\!=\!\left(R\!+\!B^\textsf{T}\!P_{k+1}B\right)^{-1}\!B^\textsf{T}\!P_{k+1}A$, and $\tilde{A}_k\!=\!A\!-\!B\tilde{B}_k$.
%\begin{align*}
%\tilde{A}_k&=A-B\tilde{B}_k,\\
%\tilde{B}_k&=\left(R+B^\textsf{T}P_{k+1}B\right)^{-1}B^\textsf{T}P_{k+1}A.
%\end{align*}
The equality (\ref{value_function1_optimal}) is ensured since
\begin{align*}
&\E[e_k^{\textsf{T}}Q_1\hat{x}_k|\mathcal{\bar{I}}_k]=\E[\hat{x}_k^{\textsf{T}}Q_1e_k|\mathcal{\bar{I}}_k]=\\
%\E\left[(x_k-\hat{x}_k)^{\textsf{T}}Q_1\hat{x}_k|\mathcal{\bar{I}}_k\right]\\
&\E[x_k^{\textsf{T}}Q_1\E[x_k|\mathcal{\bar{I}}_k]\big|\mathcal{\bar{I}}_k]\!-\!\E[(\E[x_k|\mathcal{\bar{I}}_k])^{\textsf{T}}Q_1(\E[x_k|\mathcal{\bar{I}}_k])\big|\mathcal{\bar{I}}_k]\!=\!0.
\end{align*}
%With similar reasoning, we conclude $\E\left[\hat{x}_k^{\textsf{T}}Q_1e_k|\mathcal{\bar{I}}_k\right]=0$. 
Comparing (\ref{value_function1_optimal}) with (\ref{value_function_re}), the followings are concluded
\begin{align}\nonumber
%P_k&=Q_1+\tilde{A}_k^\textsf{T}P_{k+1}\tilde{A}_k+\tilde{B}_k^\textsf{T}R\tilde{B}_k,\\\nonumber
&\pi_k=\E\!\left[e_k^{\textsf{T}}Q_1e_k+\xi_{k+1}^{\textsf{T}}P_{k+1}\xi_{k+1}+\pi_{k+1}|\mathcal{\bar{I}}_k\right]\\\label{pi_k}
&=\!\E\!\left[\!\sum\nolimits_{t=k}^{T-1}\!\!e_t^{\textsf{T}}Q_1e_t\!+\!e_T^{\textsf{T}}Q_2e_T\!+\!\sum\nolimits_{t=k+1}^T\!\xi_{t}^{\textsf{T}}P_{t}\xi_{t}|\mathcal{\bar{I}}_k\!\right]\!.
\end{align}
%and this completes the proof.
%One can show through straightforward algebraic mathematical manipulations that (\ref{cov.}) is equivalent to 
%\begin{align}\label{Riccati}
%&P_k=\\\nonumber
%&Q_1+A^\textsf{T}\!\left(\!P_{k+1}-P_{k+1}B\left(R+B^\textsf{T}P_{k+1}B\right)^{-1}\!B^\textsf{T}P_{k+1}\!\right)\!A,
%\end{align}
%and $P_T=Q_2$, which completes the proof.
From definitions of $\xi_k$ and $e_k$, it concludes for all $k\!\ge \!1$, that \vspace{-3mm} 
\begin{align*}
\xi_{k}+e_k&=x_k-\hat{x}_k^-=
%\\
%&=(Ax_{k-1}+Bu_{k-1}+w_{k-1})\\&-\E[Ax_{k-1}+Bu_{k-1}+w_{k-1}|\mathcal{\bar{I}}_{k-1}]\\
%&=
Ae_{k-1}+w_{k-1}.
\end{align*}
Knowing that $\E[\xi_k^{\textsf{T}}P_ke_k|\mathcal{\bar{I}}_k]\!=\!\xi_k^{\textsf{T}}P_k\E[e_k|\mathcal{\bar{I}}_k]\!=\!0$, we obtain
\begin{align*}
\E&[\xi_k^{\textsf{T}}P_k\xi_k|\mathcal{\bar{I}}_k]\!+\!\E[e_k^{\textsf{T}}P_ke_k|\mathcal{\bar{I}}_k]\!=\!\E[(\xi_{k}\!+\!e_k)^{\textsf{T}}P_k(\xi_{k}\!+\!e_k)|\mathcal{\bar{I}}_k] \\
&=\E[(Ae_{k-1}+w_{k-1})^{\textsf{T}}P_k(Ae_{k-1}+w_{k-1})|\mathcal{\bar{I}}_k] \\
&=\E[e_{k-1}^{\textsf{T}}A^{\textsf{T}}P_kAe_{k-1}|\mathcal{\bar{I}}_k]+tr(P_kW).
\end{align*}
Then it follows that
\begin{align*}
&\E\left[\sum\nolimits_{t=k+1}^T\!\xi_{t}^{\textsf{T}}P_{t}\xi_{t}|\mathcal{\bar{I}}_k\right]\!=\!\sum\nolimits_{t=k+1}^T\!tr(P_tW)\\
&- \!\E\!\left[\sum\nolimits_{t=k+1}^T\!\!e_{t}^{\textsf{T}}P_{t}e_{t}|\mathcal{\bar{I}}_k\!\right]
\!+\!\E\left[\sum\nolimits_{t=k+1}^T \!e_{t-1}^{\textsf{T}}A^{\textsf{T}}P_tAe_{t-1}|\mathcal{\bar{I}}_k\!\right]
\end{align*}
Finally, defining $\tilde{P}_t=Q_1 + A^{\textsf{T}}P_{t+1}A-P_t$, for all $T>t\ge k$, the expression (\ref{pi_k}) for $\pi_k$ can be re-written as
\begin{align}\nonumber
\pi_k&=\E\left[\sum\nolimits_{t=k}^{T-1}e_t^\textsf{T}(Q_1+A^\textsf{T}P_{t+1}A)e_t+e_T^\textsf{T}Q_2e_T|\mathcal{\bar{I}}_k\right]\\\nonumber
&-\E\left[\sum\nolimits_{t=k+1}^Te_t^\textsf{T}{P}_te_t|\mathcal{\bar{I}}_k\right]+\sum\nolimits_{t=k+1}^Ttr(P_tW)\\\nonumber
&=\E\!\left[e_k^\textsf{T}{P}_ke_k\!+\!\sum\nolimits_{t=k}^{T-1}e_t^\textsf{T}\tilde{P}_te_t|\mathcal{\bar{I}}_k\right]\!+\!\!\sum\nolimits_{t=k+1}^T\!tr(P_tW).
\end{align}
%where $\tilde{P}_t=Q_1 + A^{\textsf{T}}P_{t+1}A-P_t$, for all $T>t\ge k$.

\subsection{Proof of Proposition \ref{Prop_1}} \label{app:2}
Consider two cases; $k\!\ge \!D$, and $k\!< \!D$. 
%Since the smallest delay of the links is $1$,  at time $t=0$ the controller only knows the prior statistics of the initial state and hence $\E[x_0|\mathcal{\bar{I}}_0]=\E[x_0]$.
At any time $k \!\ge \!D$, the latest information the controller can have is~$x_{k-1}$, only if %the smallest delay link at time $k-1$ is selected i.e. 
$\theta_{k-1}^1\!=\!1$. If $\theta_{k-1}^1\!=\!0$, the latest information available is $x_{k-2}$, only if $\theta_{k-2}^1 \vee \theta_{k-2}^2\!=\!1$ (`$\vee$' is the logical {\tt OR} operator). The algebraic representation of the logical constraint $\theta_{k-2}^1 \vee \theta_{k-2}^2\!=\!1$ is $\theta_{k-2}^1 + \theta_{k-2}^2\!-\!\theta_{k-2}^1 \cdot\theta_{k-2}^2\!=\!1$. Similarly, we reach
\begin{align}\small\label{estimation1}
&\E[x_k|\mathcal{\bar{I}}_k]=\underbrace{\theta_{k-1}^1}_{b_{1,k}}\E[x_k|x_{k-1},U_{k-1}]+\\\nonumber
&\underbrace{\left(1-\theta_{k-1}^1\right)\!\left(\theta_{k-2}^1\!\vee\!\theta_{k-2}^2\right)}_{b_{2,k}}\!\E[x_k|x_{k-2},U_{k-1}]+\\\nonumber
&\!\underbrace{\left(1\!-\!\theta_{k-1}^1\right)\!\left(1\!-\!\theta_{k-2}^1\right)\!\left(1\!-\!\theta_{k-2}^2\right)\!\left(\!\vee_{i=1}^3\theta_{k-3}^i\!\right)}_{b_{3,k}}\!\E[x_k|x_{k-3},U_{k-1}]\\\nonumber
&+\quad\cdots\quad+\\\nonumber
&\!\underbrace{\prod\nolimits_{d=1}^{D-1}\prod\nolimits_{j=1}^{d}\left(1\!-\!\theta_{k-d}^j\right)\!\left(\vee_{i=1}^D\!\theta_{k-D}^i\right)}_{b_{D,k}}\!\!\E[x_k|x_{k-D},U_{k-1}]
\end{align}

For $k\!<\!D$, the oldest information the controller can have is $x_0$, only if $\vee_{i=1}^k\theta_0^i\!=\!1$. Otherwise, if at time $0$, the used link(s) had delay(s) greater than $k$, then $x_0$ is not available at time $k$, hence  statistics of $x_0$ are used. Thus for $k\!<\!D$,
\begin{align*}
&\E[x_k|\mathcal{\bar{I}}_k]=\theta_{k-1}^1\E[x_k|x_{k-1},U_{k-1}]+\\\nonumber
&\left(1-\theta_{k-1}^1\right)\!\left(\theta_{k-2}^1\!\vee\!\theta_{k-2}^2\right)\E[x_k|x_{k-2},U_{k-1}]+\\\nonumber
%&\!\left(1\!-\theta_{k-1}^1\right)\!\left(1\!-\theta_{k-2}^1\right)\!\left(1\!-\theta_{k-2}^2\right)\!\left(\!\vee_{i=1}^3\theta_{k-3}^i\!\right)\!\E[x_k|x_{k-3},U_{k-1}]\\\nonumber
&+\quad\cdots\quad+\\\nonumber
&\prod\nolimits_{d=1}^{k-1}\prod\nolimits_{j=1}^{d}\left(1-\theta_{k-d}^j\right)\left(\vee_{i=1}^k\theta_{0}^i\right)\E[x_k|x_0,U_{k-1}]+\\\nonumber
&\prod\nolimits_{d=1}^{k}\prod\nolimits_{j=1}^{d}\left(1-\theta_{k-d}^j\right)\E[x_k|\bar{\mathcal{I}_0},U_{k-1}]
\end{align*}

%For all $k\ge D$, let us denote
%\begin{align*}
%b_{1,k}&=\theta_{k-1}^1,\\
%b_{2,k}&=\left(1-\theta_{k-1}^1\right)\!\left(\theta_{k-2}^1\!\vee\!\theta_{k-2}^2\right),\\
%b_{3,k}&=\left(1\!-\theta_{k-1}^1\right)\!\left(1\!-\theta_{k-2}^1\right)\!\left(1\!-\theta_{k-2}^2\right)\!\left(\!\vee_{i=1}^3\theta_{k-3}^i\!\right),\\
%&\quad \!\!\!\ldots\\
%b_{D,k}&=\prod\nolimits_{d=1}^{D-1}\prod\nolimits_{j=1}^{d}\left(1-\theta_{k-d}^j\right)\left(\vee_{i=1}^D\theta_{k-D}^i\right).
%\end{align*}
For $k\!<\!D$, the same definition of $b_{1,k},b_{2,k},\cdots,b_{k,k}$ under-braced in (\ref{estimation1}) is used, while in addition, we define $b_{k+1,k}\!=\!\prod_{d=1}^{k}\prod_{j=1}^{d}(1-\theta_{k-d}^j)$ and $b_{k+2,k}\!=\!b_{k+3,k}\!=\!\cdots\!=\!b_{D,k}\!=\!0$. Finally, employing $\E[x_k|x_{-1},U_{k-1}]\triangleq\E[x_k|\mathcal{\bar{I}}_0,U_{k-1}]$, the proof then readily follows.